\documentclass[12pt,twoside]{amsart}

\usepackage{setspace, a4wide}
\usepackage[english]{babel}
\usepackage{amssymb,amsfonts,amsthm,amsmath} 
\usepackage{layout,version,enumerate}
\usepackage[all,matrix]{xy}
\usepackage[colorlinks=true,linkcolor=blue,citecolor=magenta, pagebackref]{hyperref}
\usepackage[numbers]{natbib}

\usepackage[usenames]{color}

\theoremstyle{plain}
\newtheorem{thm}{Theorem}
\newtheorem{prop}[thm]{Proposition}
\newtheorem{crl}[thm]{Corollary}

\theoremstyle{definition}

\theoremstyle{remark}
\newtheorem{exm}[thm]{Example}

\usepackage{tikz}
\usetikzlibrary{matrix,arrows}

\begin{document}


\thispagestyle{empty}
\title{About the solvability of matrix polynomial equations}

\author{Tim Netzer}
\address{T.N., Universit\"at Innsbruck, 6020 Innsbruck, Austria}
\email{tim.netzer@uibk.ac.at}
\author{Andreas Thom  }
\address{A.T., TU Dresden, 01062 Dresden, Germany}
\email{andreas.thom@tu-dresden.de}
\begin{abstract}
We study self-adjoint matrix polynomial equations in a single variable and prove existence of self-adjoint solutions under some assumptions on the leading form. Our main result is that any self-adjoint matrix polynomial equation of odd degree with non-degenerate leading form can be solved in self-adjoint matrices. We also study equations of even degree and equations in many variables.
\end{abstract}

\maketitle


Let $A,B,C,D,E$ be self-adjoint $(n \times n)$-matrices. Is there a self-adjoint $(n \times n)$-matrix $X$ such that equation
$$XAXAX + X^2 + BXCXB + DXD -X + E = 0$$ holds?
Even though this is a problem that belongs to basic algebra, there seems to be no satisfactory theory that studies the solvability of matrix polynomial equations in this generality. We will come back to this specific question after we stated Corollary \ref{main2}.

\vspace{0.2cm}

In this note, we want to study some aspects of the theory of  matrix polynomial equations. Throughout this paper let $n \in \mathbb N$ be fixed. We denote by $M_n (\mathbb C)$ the algebra of $(n\times n)$-matrices over $\mathbb C$ and study the \emph{polynomial algebra} $${\rm Pol}(M_n (\mathbb C)) = \bigoplus_{d=0}^{\infty} M_n(\mathbb C)^{\otimes d+1}$$
endowed with usual addition and multiplication given by
$$(A_0 \otimes \cdots \otimes A_d)\cdot (B_0 \otimes  \cdots \otimes B_{d'})= A_0 \otimes \cdots \otimes A_{d-1} \otimes A_dB_0 \otimes B_1 \otimes \cdots \otimes B_{d'}.$$
This way, ${\rm Pol}(M_n (\mathbb C))$ is a graded unital ring and we may homomorphically map the ring ${\rm Pol}(M_n (\mathbb C))$ to the ring of ($n \times n$)-matrix polynomials over the vector space $M_n(\mathbb C)$ with multiplication given by matrix multiplication. Indeed, it is easy to see that a natural homorphism of rings is given by
$${\rm Pol}(M_n (\mathbb C)) \ni A_0 \otimes \cdots \otimes A_d \mapsto \{M_n(\mathbb C) \ni X \mapsto A_0XA_1 \cdots A_{d-1}XA_d \in M_n(\mathbb C)\}.$$ Note that there are non-trivial elements, for example $e_{11} \otimes e_{22} \otimes e_{11} - e_{12} \otimes e_{11} \otimes e_{21} \in {\rm Pol}(M_n(\mathbb C))$ that give rise to the trivial polynomial map on $M_n(\mathbb C)$.
However, by abuse of notation, we will sometimes identify an element in ${\rm Pol}(M_n (\mathbb C))$ with the associated polynomial map on $M_n(\mathbb C)$.

We will be concerned with the question of the solvability of self-adjoint polynomial equation associated with $p \in {\rm Pol}(M_n (\mathbb C))$ in self-adjoint matrices. In general, this is a very complicated question, but we can make some progress assuming additional conditions on the leading form. We say that $p \in {\rm Pol}(M_n (\mathbb C))$ is of degree $d$ if its highest degree contribution is of degree $d$, where we agree for convenience that the zero polynomial is of degree $- \infty$. The leading form of a polynomial $p \in {\rm Pol}(M_n (\mathbb C))$ of degree $d$ is its component in degree $d$, which we denote by $p^{(d)}$. The leading form of any non-zero polynomial is a homogeneous matrix polynomial. 

We also consider the involution on ${\rm Pol}(M_n (\mathbb C))$ which extends the adjoint $A \mapsto A^*$ on the part of degree zero. More concretely, the involution is given by
$$A_0 \otimes \cdots \otimes A_d \mapsto A_d^* \otimes \cdots \otimes A_0^*$$
on the part of degree $d$. We say that a matrix polynomial $p \in {\rm Pol}(M_n (\mathbb C))$ is self-adjoint if $p^*=p$. The definition of the involution on ${\rm Pol}(M_n (\mathbb C))$ is made in such a way that $p(X)^*=p^*(X^*)$ for all $X \in M_n(\mathbb C)$. In particular, if $X$ and $p$ are self-adjoint, then $p(X)$ is self-adjoint as well. We denote the real-linear subspace of self-adjoint matrices by $M_n(\mathbb C)_{h}$. A polynomial $p \in {\rm Pol}(M_n (\mathbb C))$ is self-adjoint if is has real coefficients when expressed as a polynomial in the natural coordinates on $M_n(\mathbb C)_{h}$.

We say that a self-adjoint homogeneous polynomial $p \in {\rm Pol}(M_n(\mathbb C))$ is non-degenerate if $p(X)=0$ for $X=X^*$ only if $X=0$. 
In our study, the unit sphere 
$${\mathbb S}(M_n(\mathbb C)_{h}) = \left\{X \in M_n(\mathbb C)_{h} \mid \sum_{i,j=1}^n |x_{ij}|^2 = n \right\}$$ with respect to the norm  where $\|X\|^2= \frac1n  \sum_{i,j=1}^n |x_{ij}|^2 $ will play a major role. Note that each homogeneous and non-degenerate polynomial map $p \colon M_n(\mathbb C)_{h} \to M_n(\mathbb C)_{h}$ induces a continuous self-map $p_{\mathbb S}$ of ${\mathbb S}(M_n(\mathbb C)_{h})$ given by $p_{\mathbb S}(X)=p(X)/\|p(X)\|$. Note that ${\mathbb S}(M_n(\mathbb C)_{h})$ is just the $(n^2-1)$-dimensional sphere, in particular, it is an oriented topological manifold and continuous self-maps admit a topological degree (or Brouwer degree) as an integer-valued invariant depending only on the homotopy class of continuous maps. Our main result can be formulated as follows.

\begin{thm} \label{main}
Let $p \in {\rm Pol}(M_n (\mathbb C))$ be a self-adjoint matrix polynomial of degree $d$ such that
\begin{enumerate}
\item[(i)] $p^{(d)}$ is non-degenerate, and
\item[(ii)] the induced map $p^{(d)}_{\mathbb S} \colon {\mathbb S}(M_n(\mathbb C)_{h}) \to {\mathbb S}(M_n(\mathbb C)_{h})$ has non-zero topological degree.
\end{enumerate}
Then, there exists $X \in M_n(\mathbb C)_{h}$ such that $p(X)=0$.
\end{thm}
\begin{proof}
Assume that there is no $X \in M_n(\mathbb C)_{h}$ such that $p(X)=0$. Then, for $t \in [0,\infty)$, we define a continuous family of continuous maps $p_t \colon {\mathbb S}(M_n(\mathbb C)_{h}) \to  {\mathbb S}(M_n(\mathbb C)_{h})$ by the formula
$$p_t(X) = \frac{p(tX)}{\|p(tX)\|}.$$ This defines a homotopy which has a limit as $t$ tends to infinity, in fact it is easy to see that
$$\lim_{t \to \infty} p_t(X)= p_{\mathbb S}^{(d)}(X).$$ Since $p_0$ is constant and $p^{(d)}_{\mathbb S}$ is assumed to have non-zero topological degree, we obtain a contradiction.
\end{proof}

The previous result can be compared to the proof of the fundamental theorem of algebra using basic notions of algebraic topology that became known in the 1930's and various extensions of it to the quaternions \cite{en1, ni1} and octonions \cite{rod}. Despite the vivid interest in polynomial equations in control theory, an extension of the methods to matrix algebras has to the best of our knowledge previously not been observed. In the next part of the article we will be concerned with computations of the topological degree of the map associated with the leading form of a matrix polynomial.

\begin{prop} \label{prop1}
For any odd $d \in \mathbb N$ and $n \geq 2$, the map $M_n(\mathbb C) \ni X \mapsto X^d \in M_n(\mathbb C)$ induces a continuous map of topological degree one on ${\mathbb S}(M_n(\mathbb C)_{h})$.
\end{prop}
\begin{proof}
It is easy to see that for odd $d$, the map $X \mapsto X^d$ is homeomorphism of $M_n(\mathbb C)_h$. Moreover the tangent map of the induced map on ${\mathbb S}(M_n(\mathbb C)_{h})$ at $X=1_n$ clearly preserves the orientation and thus $X \mapsto X^d$ is of topological degree one.
\end{proof}

In general, if the leading form is not given by a monomial expression, then it is hard to compute the topological degree exactly. However, if $d$ is odd, then non-degeneracy implies that the topological degree must be always non-zero.

\begin{prop} \label{prop2}
Let $d \in \mathbb N$ be odd and $n \geq 2$. Let $p \in {\rm Pol}(M_n(\mathbb C))$ be a self-adjoint, homogeneous and non-degenerate polynomial of degree $d$. Then, $p_{\mathbb S} \colon {\mathbb S}(M_n(\mathbb C)_{h}) \to {\mathbb S}(M_n(\mathbb C)_{h})$ has non-zero topological degree.
\end{prop}
\begin{proof}
The associated real projective space ${\mathbb P}_{\mathbb R}(M_n(\mathbb C)_{h})$ is a closed topological manifold whose top-dimensional homology group is isomorphic to $\mathbb Z$ if $n$ is odd and $\mathbb Z/2 \mathbb Z$ if $n$ is even. Thus, every continuous self-map of ${\mathbb P}_{\mathbb R}(M_n(\mathbb C)_{h})$ admits a $(\mathbb Z/2 \mathbb Z)$-valued topological degree, which is the reduction modulo two of the topological degree of the map on spheres. Again, the $(\mathbb Z/2 \mathbb Z)$-valued topological degree only depends on the homotopy class of continuous maps. If there are no multiplicities, then this topological degree on the level of real projective spaces is equal to the number of generic pre-images modulo two. In our situation, since we assume that the homogeneous part of highest degree is non-degenerate, we have $n^2$ homogeneous real polynomials $(p_{ij})_{i,j=1}^n$ of the same degree $d$ that do not have a common non-trivial real zero. A priori, this does not mean that they could not have a non-trivial common complex zero. However, we claim that we can make an arbitrarily small real perturbation $(p_{t,ij})_{i,j=1}^n$ for $t \in [0,1]$ such that $p_{0,ij} = p_{ij}$, $t \mapsto p_{t,ij}$ is continuous for $1 \leq i,j \leq n$, and $(p_{t,ij})_{i,j=1}^n$ has no common non-trivial complex zero for all $t \in (0,1]$. Indeed, the space of $n^2$-tuples of homogeneous polynomials of the same degree $d$ form a vector space and a family $(q_{ij})_{i,j=1}^n$ has a common complex zero if and only if the multipolynomial resultant, an integer polynomial in the coefficients of the polynomials $(q_{ij})_{i,j=1}^n$ vanishes. See \cite{gkz, wz} for a definition and a detailed discussion of the properties of the multipolynomial resultant. Since the multipolynomial resultant is not entirely zero, it cannot vanish on \emph{all} real polynomials as they form a Zariski dense subset. Thus, the set of real polynomials for which it is zero is of codimension one in the space of real polynomials. It now follows easily that one can perturb the given family $(p_{ij})_{i,j=1}^n$ in the way described above. Note that $(p_{t,ij})_{i,j=1}^n$ for $t \in [0,1]$ is a continuous family of continuous self-maps of the real projective space of the same topological degree.

Thus, if we are only interested in non-vanishing of the topological degree of the self-map of ${\mathbb P}_{\mathbb R}(M_n(\mathbb C)_{h})$, we may assume without loss of generality that the family $(p_{ij})_{i,j=1}^n$ has no non-trivial common complex zero. Now, consider also the associated polynomial map 
$$\bar p_{\mathbb S} \colon {\mathbb P}_{\mathbb C}(M_n(\mathbb C)) \to {\mathbb P}_{\mathbb C}(M_n(\mathbb C))$$ on the complex projective space. This is a polynomial map given by $n^2$ homogeneous polynomials of degree $d$ and thus, a generic point in ${\mathbb P}_{\mathbb C}(M_n(\mathbb C))$ will have $d^{n^2-1}$ pre-images without multiplicities. Again, using an analogous argument as above, this will also be true for generic points in ${\mathbb P}_{\mathbb R}(M_n(\mathbb C)_h)$. In particular, since $d$ is odd this number is odd. Since $p$ is self-adjoint, the non-real pre-images of a real point come in conjugate pairs. Thus, we can conclude that the number of real pre-images of $\bar p_{\mathbb S}$ of a point in ${\mathbb P}_{\mathbb R}(M_n(\mathbb C)_{h})$ is always odd without multiplicities. Hence, the continuous map
$\bar p_{\mathbb S} \colon {\mathbb P}_{\mathbb R}(M_n(\mathbb C)_{h}) \to {\mathbb P}_{\mathbb R}(M_n(\mathbb C)_{h})$ has a non-zero $(\mathbb Z/2\mathbb Z)$-valued topological degree. We conclude that the topological degree of $p_{\mathbb S}$ on the level of spheres must be non-zero. This finishes the proof.
\end{proof}

Note that in general a polynomial $p \in {\rm Pol}(M_n(\mathbb C))$ homogeneous of degree $d$ will  -- even if non-degenerate in our sense -- only determine a rational self-map of the complex projective space ${\mathbb P}_{\mathbb C}(M_n(\mathbb C))$ and there is no reason to think that it would have degree $d^{n^2-1}$. Indeed, consider the map $X \mapsto X^d$ on $M_n(\mathbb C)$. A generic $(n \times n)$-matrix has  exactly $d^n$ different $d$-th roots and the generic number of pre-images as a rational self-map of ${\mathbb P}_{\mathbb C}(M_n(\mathbb C))$ is $d^{n-1}$. This phenomenon is due to the non-triviality of the singularity $\{X \in M_n (\mathbb C) \mid X^d=0\}$ -- in this case it not even zero-dimensional. However, enough \emph{oddity} survives for our arguments to go through also in this case -- as we seen in Proposition \ref{prop1}, in this case the topological degree is equal to 1 for all $d \in \mathbb N$.

\begin{crl} \label{main2}
Let $d \in \mathbb N$ be odd and let $p \in  {\rm Pol}(M_n (\mathbb C))$ be a self-adjoint matrix polynomial of degree $d$ and assume that $p^{(d)}$ is non-degenerate. Then, there exists a self-adjoint matrix $X \in M_n(\mathbb C)_h$ such that $p(X)=0.$
\end{crl}
We can now partially answer the question from the beginning of this article.
The leading form of the matrix polynomial
$p(X) = XAXAX + X^2 + BXCXB + DXD -X + E$ is $p^{(3)}(X)=XAXAX$. Thus, Corollary \ref{main2} applies if the equation $XAXAX=0$ implies that $X=0$ for self-adjoint $X \in M_n(\mathbb C)$. This holds for example when $A$ is positive definite. Indeed, we easily compute:
$$XAXAX=0 \Rightarrow A^{1/2}XAXAXA^{1/2}=0 \Rightarrow (A^{1/2}XA^{1/2})^3=0 \Rightarrow A^{1/2}XA^{1/2} \Rightarrow X=0.$$
On the other side, the equation $XAXAX+1_n=0$ implies that $A$ is invertible if a solution exists. Thus, some non-degeneracy condition on the leading form is necessary. Whenever the matrix $A$ is indefinite, one can show that the leading form $XAXAX$ is degenerate and also that the equation $XAXAX + E=0$  cannot always be solved in self-adjoint matrices -- thus, non-degeneracy of the leading form is necessary and sufficient in this special case. However, in general, it is unclear if (for $d$ odd) non-degeneracy is the minimal assumption on the leading form to ensure existence for solution for all self-adjoint polynomials with this leading form. An obvious necessary condition is that the leading form be surjective as a map on self-adjoint matrices.

If $p =p^*\in {\rm Pol}(M_n (\mathbb C))$ is of even degree equal to $d$, then the map $p^{(d)}_{\mathbb S}$ will factor through the associated real projective space. Thus, the topological degree of the induced map on spheres will be zero if $n$ is odd. Indeed, if $n$ is odd, then the $(n^2-1)$-dimensional projective space is not orientable and hence degrees of maps must vanish. However, if $n$ is even, the degree of the associated map on spheres can be a non-zero even integer. Indeed, one can show for example that the quadratic map
$$\left(\begin{matrix} a_{11} & a_{12} \\ a_{21} & a_{22} \end{matrix} \right) \mapsto \left(\begin{matrix} a_{11}^2 - a^2_{22} - a_{12}a_{21} & 2 a_{11} a_{12} \\ 2 a_{11} a_{21} & a_{11}a_{22}\end{matrix} \right)$$
is non-degenerate and induces a map of degree two on the three-dimensional unit-sphere of $M_2(\mathbb C)_h$. The associated self-adjoint element in ${\rm Pol}(M_2(\mathbb C))$ is
$$e_{11} \otimes e_{11} \otimes e_{11} - e_{12} \otimes e_{22} \ \otimes e_{21} - e_{11} \otimes e_{22} \otimes e_{11} + 2 e_{11} \otimes e_{11} \otimes e_{22} $$
$$+ \ 2 e_{22} \otimes e_{11} \otimes e_{11} + \frac12(e_{21} \otimes e_{12} \otimes e_{22} + e_{22} \otimes e_{21} \otimes e_{12}),$$
where $e_{ij}$ for $1 \leq i,j \leq 2$ denote the standard matrix units in $M_2(\mathbb C)$. Thus any self-adjoint quadratic equation with  leading term like this can be solved in self-adjoint matrices according to Theorem \ref{main}.

In general, polynomial self-maps  of odd-dimensional sphere of topological degree $d$ were constructed by Wood \cite{wood} using homogeneous polynomials of degree $|d|$. For even-dimensional spheres and $d$ odd there is a construction by Turiel \cite{turiel} using homogeneous polynomials of degree $2|d|-1$.

\vspace{0.2cm}

We want to spend the rest of the article to explain how the results generalize to the case of many variables. There are various ways to apply the general technique -- however, we will focus only on one possibility. Let $k \in \mathbb N$ and denote by $\mathbb C\langle x_1,\dots,x_k \rangle$ the free unital $\mathbb C$-algebra with generators $x_1,\dots,x_k$. A free matrix polynomial in $k$ variables is an element in the unital ring 
$$M_n(\mathbb C) \ast_{\mathbb C} \mathbb C\langle x_1,\dots,x_k \rangle,$$
that arises as an amalgamated free product of rings. A typical element in the ring $M_n(\mathbb C) \ast_{\mathbb C} \mathbb C\langle x_1,\dots,x_k \rangle$ is a linear combination of monomials of the form $A_0 x_{i_1} A_1 x_{i_2} \cdots x_{i_d} A_d$ with $i_l \in \{1,\dots,k\}$ for $1 \leq l \leq d$ and $A_0,\dots,A_d \in M_n(\mathbb C)$. Note that this ring is just ${\rm Pol}(M_n(\mathbb C))$ in the case $k=1$; however note that as discussed before, some elements $p \in M_n(\mathbb C) \ast_{\mathbb C} \mathbb C\langle x \rangle$ evaluate trivially when considered as polynomial maps $X \mapsto p(X)$ on $M_n(\mathbb C)$. The adjoint, the degree, and the leading form of some element $p \in M_n(\mathbb C) \ast_{\mathbb C} \mathbb C\langle x_1,\dots,x_k \rangle$ are defined in a straightforward way -- however, as before the degree will not necessarily agree with the degree of the associated polynomial map from $M_n(\mathbb C)^k$ to $M_n(\mathbb C)$. We say that a $k$-tuple of homogeneous and self-adjoint free matrix polynomials $p_1,\dots,p_k$ is non-degenerate if for self-adjoint matrices $X_1,\dots,X_k \in M_n(\mathbb C)_h$ the implication
$$p_i(X_1,\dots,X_k)=0, \quad \forall 1 \leq i \leq k \quad \Longrightarrow \quad X_i=0, \quad \forall1 \leq i \leq k$$
holds.
The following generalization of Corollary \ref{main2} is now straightforward.

\begin{thm} \label{main3}
Let $k \in \mathbb N$ and let $p_1,\dots,p_k$ be self-adjoint free matrix polynomials of odd degree in $k$ variables. If the $k$-tuple of leading forms is non-degenerate, then there exist self-adjoint matrices $X_1,\dots,X_k \in M_n(\mathbb C)_h$, such that $p_i(X_1,\dots,X_k)=0$, for all $1 \leq i \leq k.$
\end{thm}
\begin{proof}
Let $d_i$ be the degree of the free matrix polynomial $p_i \in M_n(\mathbb C) \ast_{\mathbb C} \mathbb C\langle X_1,\dots,X_k \rangle$. By taking an appropriate power of $p_i$, we may replace $d_i$ by ${\rm lcm}(d_1, \dots,d_k)$. Thus, we may assume without loss of generality the degrees are all equal. We may now repeat the arguments that led to the proof of Corollary \ref{main2} starting with the polynomial map
$$M_n(\mathbb C)_h^k \ni (X_1,\dots,X_k) \mapsto (p_1(X_1,\dots,X_k), \dots, p_k(X_1,\dots,X_k)) \in M_n(\mathbb C)_h^k.$$
This ends the outline of the proof.
\end{proof}

\begin{exm}
The previous result applies for example whenever $p,q \in M_n(\mathbb C) \ast_{\mathbb C} \mathbb C\langle x,y \rangle$ are quadratic self-adjoint free matrix polynomials. In this case, we can conclude from Theorem \ref{main3} that there exist self-adjoint matrices $X,Y \in M_n(\mathbb C)_h$ such that:
$$X^3 + p(X,Y) = 0, \quad Y^3 + q(X,Y)=0.$$
\end{exm}

\section*{Acknowledgments}

This research was supported by ERC Starting Grant 277728. The second author wants to thank Jos\'e Burgos Gil and Carlos D'Andrea for an interesting discussion about the proof of Proposition \ref{prop2}.

\end{document}